\documentclass[11pt]{article}
\usepackage[a4paper]{geometry}
\usepackage{amsfonts,amsmath,amssymb,amsthm,bm,graphicx,hyperref}

\newcommand{\Aa}{\mathcal{A}}
\newcommand{\Cc}{\mathbb{C}}
\newcommand{\gammachanged}{\gamma_{\text{changed}}}
\newcommand{\gammat}{\tilde{\gamma}}
\newcommand{\geqs}{\geqslant}
\newcommand{\Hh}{\underline{H}}
\newcommand{\Hht}{\underline{\tilde{H}}}
\newcommand{\Ht}{\tilde{H}}
\newcommand{\Kk}{\mathcal{K}}
\newcommand{\Mm}{\mathcal{M}}
\newcommand{\leqs}{\leqslant}
\newcommand{\Pe}{\textit{P\!e}}
\renewcommand{\Re}{\mathrm{Re}}
\newcommand{\Rr}{\mathbb{R}}
\newcommand{\Rrnn}{\mathbb{R}^{n\times n}}

\newcommand{\tol}{\mathtt{tol}}

\newtheorem{lemma}{Lemma}
\newtheorem{proposition}{Proposition}
\newtheorem{remark}{Remark}

\title{An accurate restarting for shift-and-invert Krylov subspaces\\
computing matrix exponential actions\\of nonsymmetric matrices}

\author{Mikhail A. Botchev\thanks{Keldysh Institute of Applied Mathematics,
Russian Academy of Sciences, Miusskaya~Sq.~4, 125047 Moscow,
Russia and
Marchuk Institute of Numerical Mathematics, Russian Academy of Science,
Gubkina~St.~8, 119333 Moscow, Russia, \url{botchev@ya.ru}.
This work is supported by the Russian Science Foundation grant
No.~19-11-00338.}}

\begin{document}
\maketitle

\begin{abstract}
An accurate residual--time (AccuRT) restarting for computing
matrix exponential actions of nonsymmetric matrices 
by the shift-and-invert (SAI) Krylov subspace
method is proposed.  
The proposed restarting method is an extension of the recently 
proposed RT (residual--time) restarting and it is designed
to avoid a possible accuracy loss in the conventional
RT restarting.
An expensive part of the SAI Krylov method is solution of
linear systems with the shifted matrix.  
Since the AccuRT algorithm adjusts the shift value, we discuss 
how the proposed restarting can be implemented with just a 
single LU~factorization (or a preconditioner setup) of the shifted matrix.
Numerical experiments demonstrate an improved accuracy and 
efficiency of the approach.

\end{abstract}

\section{Introduction}
Computing action of the matrix exponential on a given vector
is a common task often occurring in various applications such as
time integration~\cite{HochbruckOstermann2010}, 
network analysis or 
model order reduction.  
If the matrix is large then Krylov subspace
methods form an important group of methods that are well suited
for this task~\cite{FrommerSimoncini08a}.  
Other methods for computing matrix exponential
actions for large matrices include
Chebyshev polynomials~\cite{BergamaschiVianello2000} and
the scaling and squaring method combined with Taylor series~\cite{AlmohyHigham2011}.
Krylov subspace computations of the matrix exponential and other
matrix functions has been an active research area, with many important
developments such as
rational Krylov subspace
methods~\cite{MoretNovati04,EshofHochbruck06,DruskinKinizhnermanZaslavsky2009,PhD_Guettel,Guettel12},
restarting
techniques~\cite{CelledoniMoret97,TalEzer2007,Afanasjew_ea08,Eiermann_ea2011,GuettelFrommerSchweitzer2014,JaweckiAuzingerKoch2018}
and interesting large-scale computational
applications~\cite{HochbruckLubich99b,DruskinKinizhnermanZaslavsky2009,Hochbruck_Pazur_ea2015,BornerErnstGuettel2015,BotchevHanseUppu2018}.

Restarting techniques allow to keep the number Krylov basis vectors
(i.e., the Krylov subspace dimension) bounded while trying to preserve
convergence properties of the method.
Recently proposed residual--time (RT) restarting for computing
matrix exponential actions seems to be an attractive technique~\cite{BoKn2020}.
One of its advantages is that the RT-restarted polynomial Krylov subspace
methods are guaranteed to converge to a required tolerance for
any restart length~\cite{BoKn2020}.
Another nice feature of the RT restarting is that the size of the
small projected problem which has to be solved in the course
of iterations does not grow with each restart (as is 
the case for some other restarting methods,
see e.g.~\cite[Chapter~3]{PhD_Niehoff}).
For instance, for a restart length~10
the Krylov subspace dimension and the size of the projected
problem are at most~10.
In addition, the projected problem in the RT-restarted Krylov
subspace methods is evaluation of the matrix exponential
of a small matrix (a projection of the original large matrix).
This is a simple problem which can be efficiently carried out by
one of the standard linear algebra
techniques~\cite{GolVanL,19ways,Higham_bookFM}.
In contrast, the projected problem in the so-called residual-based
restarting~\cite{CelledoniMoret97,BGH13} is a system of nonautonomous
ordinary differential equations (ODEs), see~\cite[formula(3)]{CelledoniMoret97}.
Although this ODE system is small-sized, its solution is typically more
costly and requires additional care, such as a proper choice of an ODE solver and
its parameters.

An important class of rational Krylov subspace methods for the
matrix exponential is shift-and-invert (SAI) Krylov subspace 
method~\cite{MoretNovati04,EshofHochbruck06}.
It often appears to be efficient in various applications
because it requires solution of linear systems with a single shifted
matrix only.
One of the shortcomings of the RT restarting proposed in~\cite{BoKn2020}
is that the accuracy of the SAI Krylov method combined with RT restarting
can not be guaranteed.  This paper proposes an extension of the RT restarting
for the SAI Krylov subspace method which is 
designed 
to attain a required accuracy.  The proposed technique works
for nonsymmetric matrices.  Moreover, we show how
to implement the proposed modification efficiently, so that
an LU factorization or a preconditioner setup has to be carried out only
once.

The paper is organized as follows.
Our restarting technique, which we call AccuRT (accurate residual--time
restarting) is described in Section~\ref{s:main}.
There we first give some background information on Krylov
subspace methods (Section~2.1), then we introduce
and analyze AccuRT
(Section~2.2)
and discuss how to organize solution of the shifted systems
efficiently (Section~2.3).
In Section~\ref{s:exper} numerical experiments for
two test problems are presented.
Finally, we draw some conclusions in Section~\ref{s:concl}.

\section{AccuRT: an accurate residual--time restarting}
\label{s:main}
Let, unless indicated otherwise, $\|\cdot\|$ denote the 2-norm.
Throughout the paper we assume that for $A\in\Rrnn$ holds
\begin{equation}
\label{fov}
\Re (x^*Ax)\geqslant 0, \qquad \forall x\in\Cc^n,
\end{equation}
where $\Re(z)$ denotes the real part of a complex number $z$.

\subsection{Krylov subspace methods and RT restarting}
Assume the matrix exponential action has to be computed for a given
matrix $A\in\Rrnn$, $v\in\Rr^n$ and $t>0$, i.e.,
\begin{equation}
\label{expA}
\text{compute} \quad y := \exp(-t A)v.
\end{equation}
The problem is equivalent to solving initial value problem (IVP)
\begin{equation}
\label{ivp}
y'(t)=-Ay(t), \quad y(0)=v,
\end{equation}
where we slightly abuse the notation by using $t$ for both the independent variable 
in~\eqref{ivp} and for time interval length in~\eqref{expA}.
Krylov subspace method for computing the matrix exponential action can be seen
as a Galerkin projection of IVP~\eqref{ivp} on the Krylov subspace
\begin{equation}
\label{Kk}  
\Kk_k(A,v) = \mathrm{span} (v,Av,A^2v,\dots, A^{k-1}v).
\end{equation}
First, an orthonormal basis of $\Kk_k(A,v)$ is computed by the standard 
Arnoldi (or, if $A=A^T$, by Lanczos) process~\cite{GolVanL,Parlett:book,Henk:book,SaadBook2003}
and stored as 
the columns $v_1$, $v_2$, \dots, $v_k$ of a matrix $V_k=[ v_1 \dots v_k]\in\Rr^{n\times k}$,
such that
the following Arnoldi decomposition holds:
\begin{equation}
\label{Arnoldi}
AV_k = V_{k+1} \Hh_k = V_kH_k + h_{k+1,k}v_{k+1}e_k^T, 
\end{equation}
where $e_k=(0,\dots,0,1)^T\in\Rr^k$, $\Hh_k\in\Rr^{(k+1)\times k}$ is 
an upper Hessenberg matrix and $H_k\in\Rr^{k\times k}$ contains the first $k$ rows 
of $\Hh_k$.
Then
the Krylov subspace approximation $y_k(t)\approx \exp(-t A)v$ is defined 
as~\cite{DruskinKnizh89,Knizh91,HochLub97}
\begin{equation}
\label{yk}
y_k(t) = V_k u(t),
\end{equation}
where $u(t): \Rr\rightarrow\Rr^k$ solves IVP with the projected matrix $H_k=V_k^TAV_k$
\begin{equation}
\label{ivp1}
u'(t) = -H_k u(t), \quad u(0) = \beta e_1.
\end{equation}
Here $\beta=\|v\|$ and $e_1=(1,0,\dots,0)^T\in\Rr^k$.
Note that IVP~\eqref{ivp1} is a Galerkin projection of~\eqref{ivp} on the 
Krylov subspace and that $u(t)$ can be computed as
\begin{equation}
\label{expH}
u(t) = \exp(-tH_k) \beta e_1.
\end{equation}
If $k$ is not too large, \eqref{expH} is a preferred way to compute $u$ 
which can be implemented by one of the standard dense linear algebra 
routines~\cite{Higham_bookFM}.  
Computing $\exp(-t H_k)$ is usually quicker than solving the ODE system 
in~\eqref{ivp1}, which requires choosing a suitable ODE solver (e.g., stiff or non-stiff)
and its parameters.
Krylov subspace method~\eqref{Arnoldi}--\eqref{expH} is sometimes called
polynomial Krylov subspace method to emphasize the fact that vectors
of the subspace~\eqref{Kk} are polynomials in $A$.

A natural way of controlling the error of the Krylov subspace approximation~\eqref{yk}
is to monitor the residual $r_k(t)$ of $y_k(t)$ with respect to the ODE $y'=-Ay$,
namely~\cite{CelledoniMoret97,DruskinGreenbaumKnizhnerman98,BGH13}
\begin{equation}
\label{rk1}
r_k(t) = -Ay_k(t) - y_k'(t).
\end{equation}
The residual is readily available in the course of the Arnoldi process
computations as~\cite{CelledoniMoret97,DruskinGreenbaumKnizhnerman98,BGH13}
\begin{equation}
\label{rk2}
r_k(t) = -h_{k+1,k} (e_k^T u(t)) v_{k+1},
\end{equation}
where we see that $r_k(t)$ is a scalar function times $v_{k+1}$.
Hence, $V_k^Tr_k(t)=0$ for all $t$ and~\eqref{yk} is indeed a Galerkin projection
on $\Kk_k(A,v)$. 
Some residual convergence results can be found, e.g., in~\cite{BoKn2020}.

\begin{figure}
\centering{\includegraphics[width=0.8\textwidth]{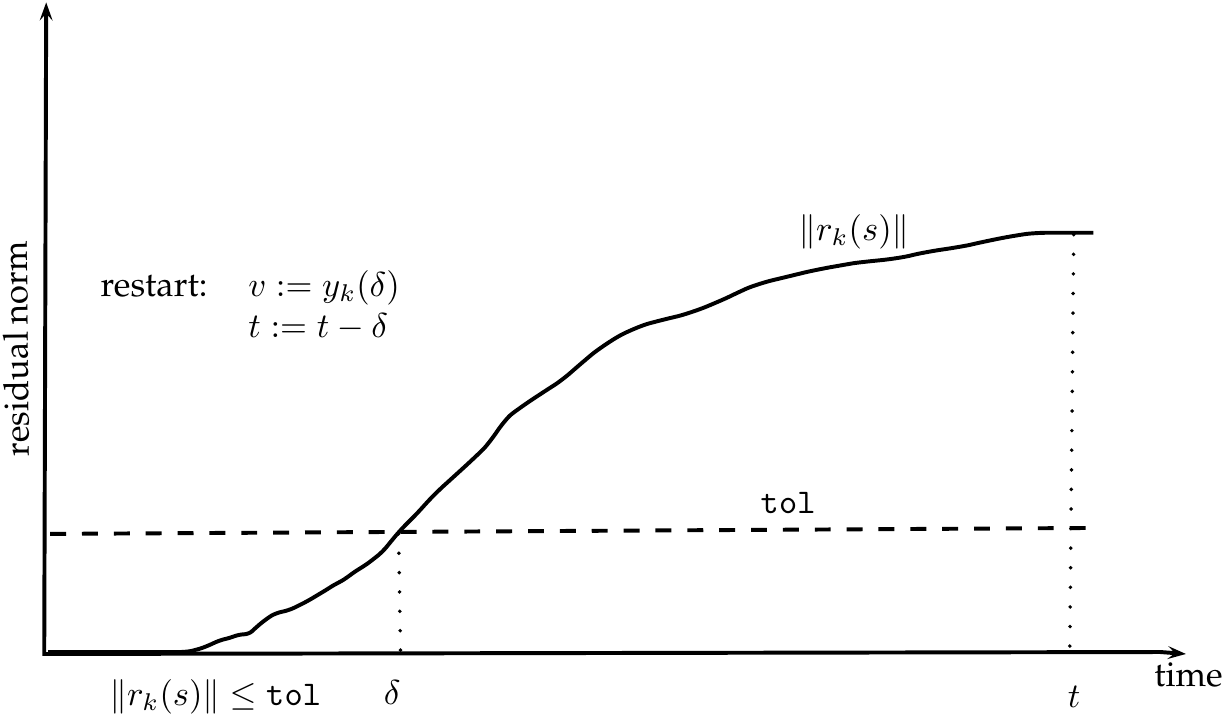}}
\caption{A sketch of the RT restarting procedure, adopted from our paper~\cite{BoKn2020}}
\label{f:RT}
\end{figure}

The SAI (shift-and-invert) Krylov subspace method~\cite{MoretNovati04,EshofHochbruck06}
for computing~\eqref{expA} 
differs from the standard polynomial Krylov subspace method described above
in that it builds the Krylov subspace for the shifted-and-inverted matrix
$(I+\gamma A)^{-1}$ instead of $A$, with $\gamma>0$
being a fixed parameter.  This is done to accelerate convergence: 
the Arnoldi process tends to emphasize the largest
eigenvalues of $(I+\gamma A)^{-1}$ corresponding to the smallest eigenvalues of 
$A$.  The latter are important for the exponential which is a quickly decaying
function (assuming real arguments).  The price to pay for this accelerated
convergence is the necessity to solve a linear system with the matrix $I+\gamma A$
at each Krylov step.
The Arnoldi decomposition~\eqref{Arnoldi} for the SAI matrix
$(I+\gamma A)^{-1}$, i.e.,
$$
(I+\gamma A)^{-1}V_k = V_{k+1} \Hht_k = V_k\tilde{H}_k + \tilde{h}_{k+1,k}v_{k+1}e_k^T,   
$$
is more convenient to use in a transformed form
\begin{equation}
\label{Arn_sai}
AV_k = V_kH_k - \frac{\tilde{h}_{k+1,k}}{\gamma}(I+\gamma A)v_{k+1}e_k^T\tilde{H}_k^{-1},
\end{equation}
where the notation $\tilde{\cdot}$ is used to indicate that the projection
is built for the SAI matrix $(I+\gamma A)^{-1}$ and $H_k$ is defined
as the inverse SAI transformation
\begin{equation}
\label{Hk_sai}  
H_k = \frac1\gamma (\tilde{H}_k^{-1}-I).
\end{equation}
Note that the matrices $V_{k+1}$ and $H_k$ here are different
from those in~\eqref{Arnoldi}. 
The SAI Krylov subspace method is analyzed, e.g., 
in~\cite{MoretNovati04,EshofHochbruck06,PhD_Guettel}. 

In the SAI Krylov subspace method the residual can be easily computed
as~\cite{BGH13}
\begin{equation}
\label{rk_sai}
r_k(t) = \frac{\tilde{h}_{k+1,k}}{\gamma}(e_k^T\tilde{H}_k^{-1}u(t))\cdot
(I+\gamma A)v_{k+1},
\end{equation}
where $u(t)$ is the solution of the projected IVP~\eqref{ivp1},
with $H_k$ being the back transformed SAI matrix~\eqref{Hk_sai}.

The shift value $\gamma$ is usually chosen in accordance with the length $t$
of the time interval $[0,t]$,
see~\cite{EshofHochbruck06}.
A possible often used value is $\gamma=t/10$.
Hence, changing $\gamma$ means changing $t$.  The usual polynomial Krylov
subspace method~\eqref{Arnoldi}--\eqref{expH} has an attractive property 
that it is invariant of $t$: once $V_{k+1}$ and $\Hh_k$ are computed, they
can be successfully used for any $t$ (although the quality of approximation
$y_k(t)\approx y(t)$ does deteriorates with $t$).
Unfortunately, this nice property is not fully shared by the SAI Krylov subspace
method: its matrices $V_{k+1}$ and $\tilde{\Hh}_k$ do depend on $\gamma$
which, in turn, depends on $t$.
However, in practice, one certainly can use the computed Arnoldi
matrices $V_{k+1}$ and $\tilde{\Hh}_k$ for a certain range of $t$
without recomputing them.

Recently proposed RT (residual--time) restarting is based on the fact that 
the residual as a function of $t$ tends to be a non-decreasing function.
Hence, once a maximum number $k_{\max}$ Krylov steps are done (so that storing
and working with more Krylov basis vectors is too expensive),
we can find a time subinterval $[0,\delta]$, $\delta<t$, where the residual
norm is already sufficiently small.  We can then restart the Krylov method
by setting $v:=y_{k_{\max}}(\delta)$, decreasing the time interval $t:=t-\delta$ 
and performing the next $k_{\max}$ Krylov steps solving problem~\eqref{expA}
with the modified $v$ and $t$.
The RT restarting procedure is sketched in Figure~\ref{f:RT}.

\begin{figure}
\includegraphics[width=0.47\textwidth]{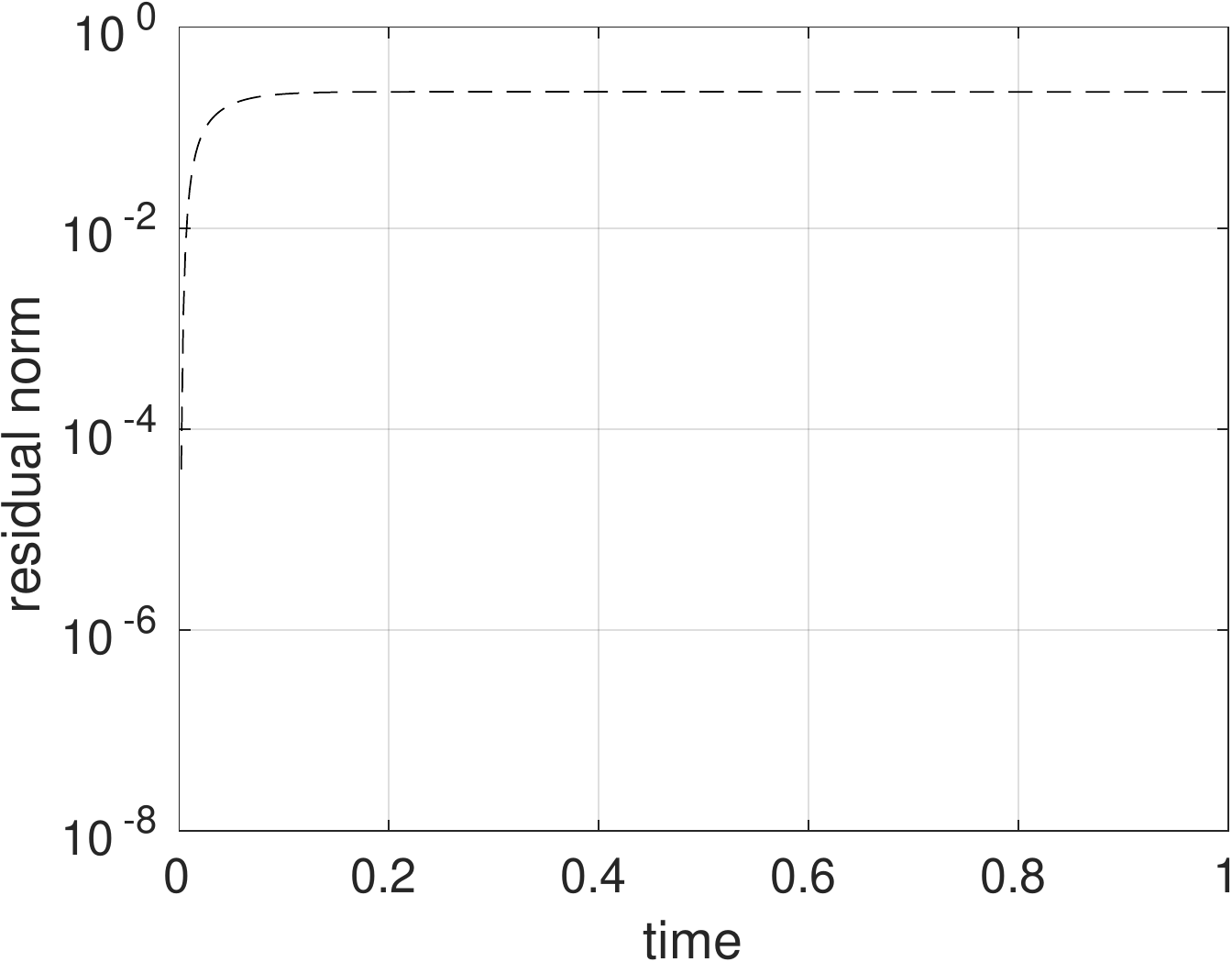}~~~%
\includegraphics[width=0.47\textwidth]{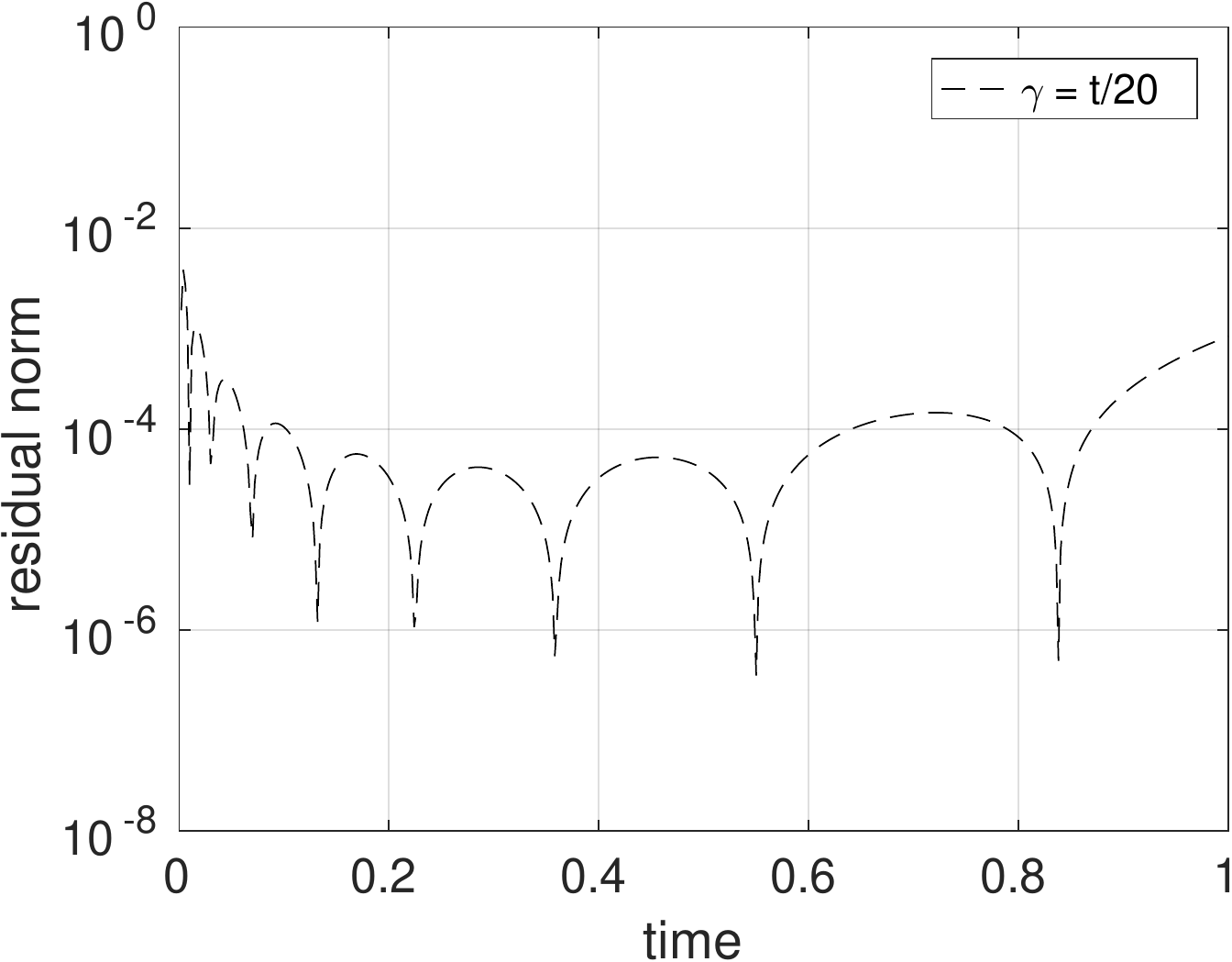}
\caption{Residual norms $\|r_k(s)\|$ as functions of time $s\in[0,t]$, $t=1$, for the polynomial
  (left) and SAI (right) Krylov subspace methods, after $k=10$ Krylov steps.
  The SAI shift value is $\gamma=t/20$.
  The matrix $A$ is a discretized
  convection--diffusion operator for the
  Peclet number $\Pe=200$ and $802\times 802$ grid (see Section~\ref{s:t1}
  for more detail). For both plots the residual norm values are computed
  in 500 equidistant points in $[0,1]$.}
\label{f:rk}  
\end{figure}

\begin{figure}
\includegraphics[width=0.47\textwidth]{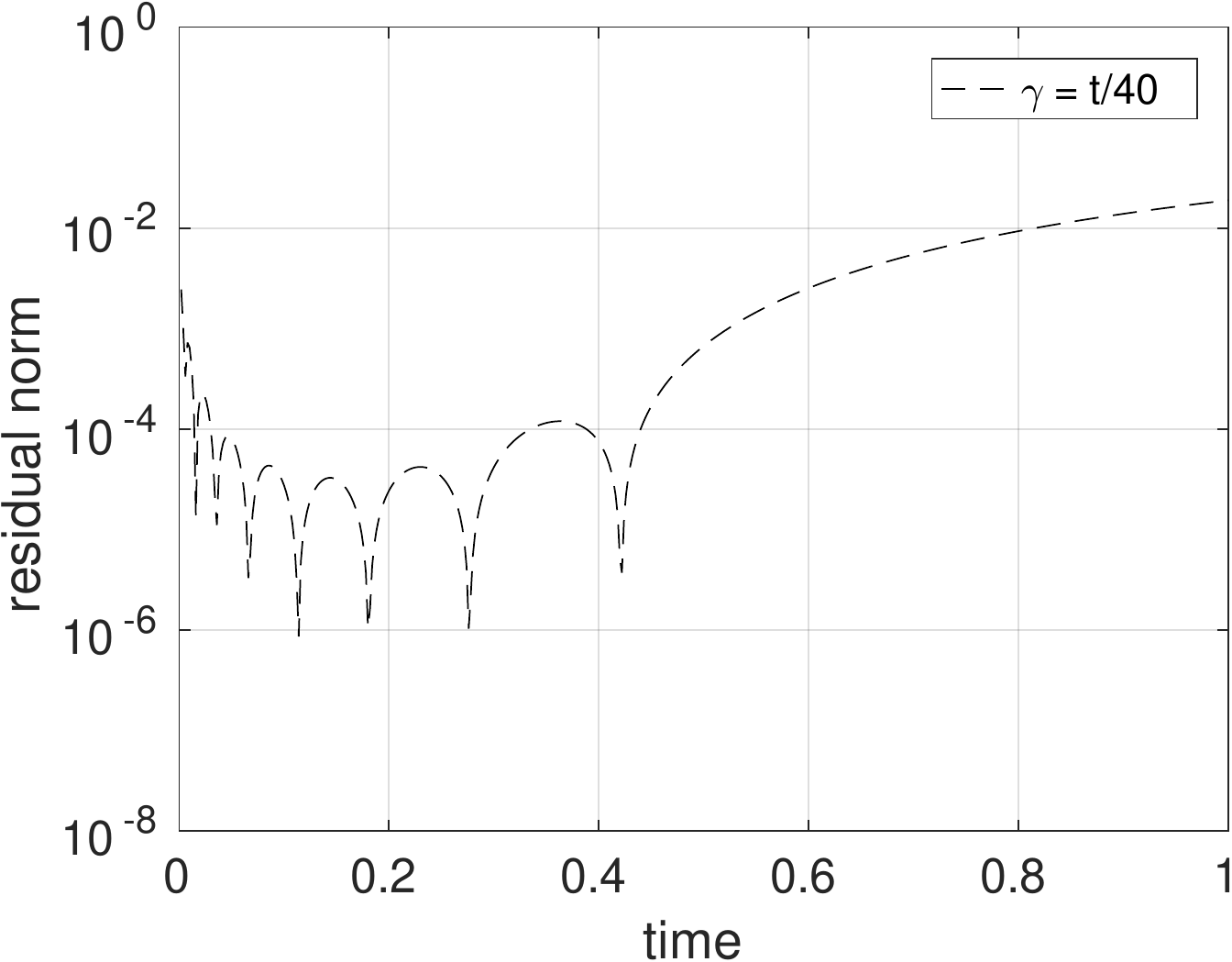}~~~%
\includegraphics[width=0.47\textwidth]{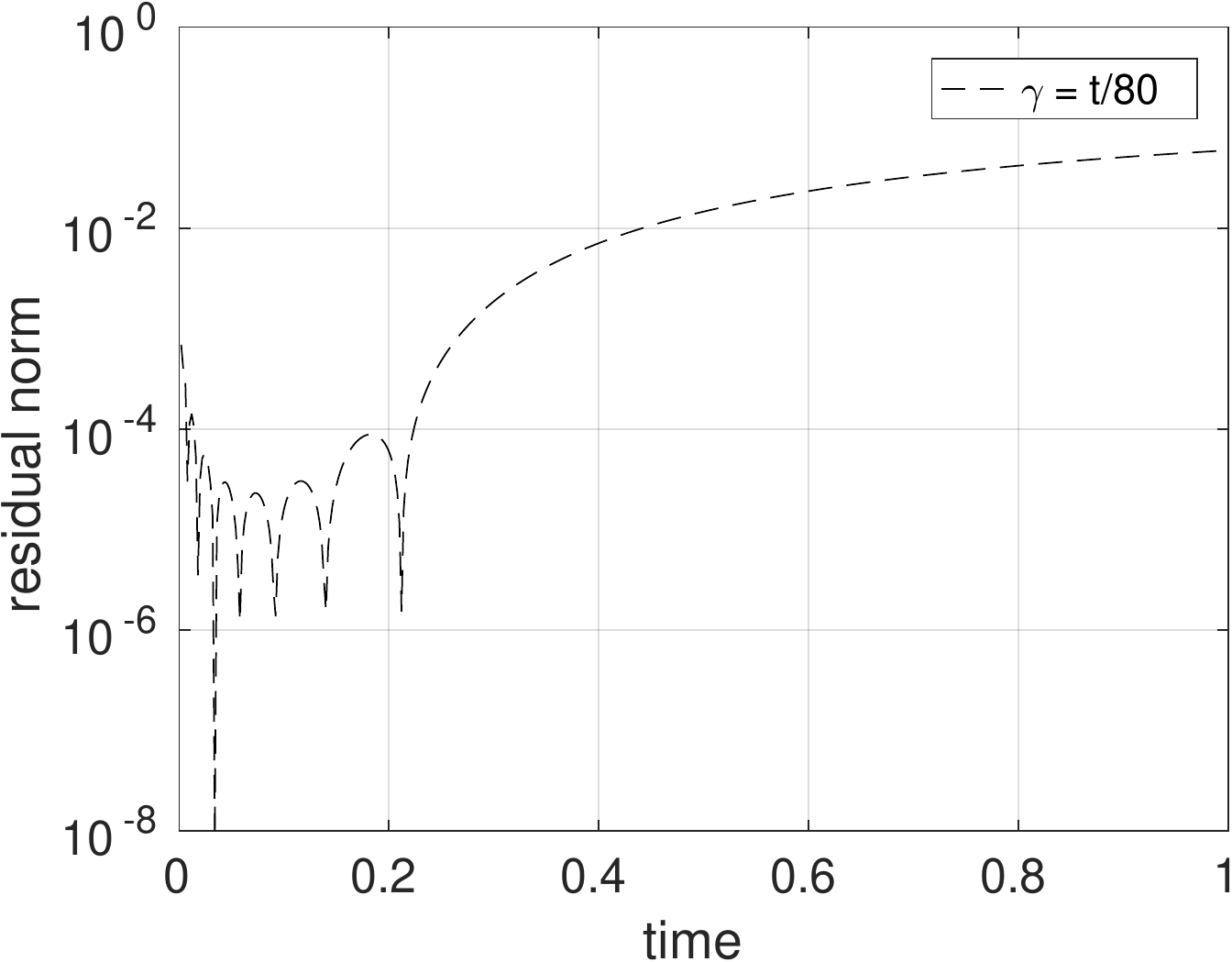}
\caption{Residual norms $\|r_k(s)\|$ as functions of time $s\in[0,t]$, $t=1$, for the SAI
  Krylov subspace methods for $\gamma=t/40$ (left) and $\gamma=t/80$ (right),
  after $k=10$ Krylov steps.
  The matrix $A$ is a discretized
  convection--diffusion operator for the
  Peclet number $\Pe=200$ and $802\times 802$ grid (see Section~\ref{s:t1}
  for more detail). For both plots the residual norm values are computed
  in 500 equidistant points in $[0,1]$.}
\label{f:rk1}  
\end{figure}

\subsection{AccuRT ideas and algorithm}
The residual as a function of $t$ in the SAI Krylov subspace method 
exhibits a much more irregular behavior than in the polynomial Krylov 
method~\eqref{Arnoldi}--\eqref{expH}, see Figure~\ref{f:rk}.
If the RT restarting is applied for the SAI Krylov subspace
method then it can happen that $\delta$ such that $\|r_k(s)\|$ is below
a certain tolerance for $s\in[0,\delta]$ is too small to be used
in practice for an efficient restarting.  Of course, we could also restart
by setting $\delta$ to any as large as possible point $s$ where $\|r_k(s)\|$
is small enough,
cf.~Figure~\ref{f:rk}.  However, there is no guarantee that
$\min_{s\in[0,t]}\|r_k(s)\|$ is within the required tolerance and, if this
is the case, the restarting with any $\delta\in[0,t]$
inevitably leads to an accuracy loss.

Here we propose an approach to fix this failure of the RT
restarting in the SAI Krylov subspace method.
The approach is based on the following two simple observations.

\begin{enumerate}
\item 
Since $\gamma$ is usually chosen proportionally to $t$, 
taking a smaller shift value $\gamma$ effectively means a shorter
time interval $[0,t]$.  For nonsymmetric matrices $A$
the SAI Krylov residual $r_k(s)$ tends to become smaller in
norm with smaller $\gamma$ on some time subinterval
$s\in[0,\delta]$, $0<\delta<t$, see Figure~\ref{f:rk1}.
\item
As already noted, to change $\gamma$ in the SAI Krylov subspace
method we have to carry out the whole Arnoldi process anew. 
However, once a shifted linear system with the matrix $I+\gamma A$
is solved for a certain shift $\gamma$,
a part of the spent computational
costs can be re-used for solving the shifted linear systems
with a smaller shift $\gammat\leqs\gamma$.
In particular, if 
a (sparse) LU factorization is computed for a certain shift
$\gamma$, it can be successfully used as a preconditioner
for solving shifted systems with $I+\gammat A$,
$\gammat\leqs\gamma$ (see Proposition~\ref{sai_syst}).
\end{enumerate}
Based on these observations we propose to organize an improved
RT restarting for the SAI Krylov subspace method as follows.
Assume we can carry out at most $k_{\max}$ steps of the Arnoldi or Lanczos
process, since storing (or working with) more than $k_{\max}$
Krylov basis vectors is too expensive.
We perform $k=1,\dots,k_{\max}$ steps checking at each step
the residual norm $\|r_k(t)\|$, cf.~\eqref{rk_sai}.
If the residual norm turns out to be smaller than the
required tolerance, we stop.  Otherwise, after performing
step~$k=k_{\max}$, we analyze the function $\|r_k(s)\|$ for $s\in[0,t]$.
If no point $s$ can be found where $\|r_k(s)\|$ is small enough,
we decrease $\gamma$ by a factor of two and repeat
$k=1,\dots,k_{\max}$ steps of the method.  The restarting
procedure (which is presented in detail in Figure~\ref{f:alg})
can be repeated until the residual norm is small enough
at the end of the given time interval.

If a (sparse) LU factorization is prohibitively expensive,
a preconditioned solver can be used to solve
the shifted linear system.  In the algorithm presented
in Figure~\ref{f:alg} we then replace the LU factorization
step by a setup of the preconditioner.
The computed preconditioner can then be used for all the shift
values, i.e., it suffices to set up the preconditioner once.
Note that the observed behavior of the SAI Krylov residual
for decreasing $\gamma$ does not hold for symmetric $A$.
This is reflected by a rather sharp
estimate~\cite[Lemma~3.1]{EshofHochbruck06}
(in the estimate there, take $\mu=0$ and set $\gamma$ proportional to $t=\tau$).

The following lemma and proposition show that the residual $r_k(s)$
of the SAI Krylov subspace method is bounded in norm as a function of
time.  The way it is bounded depends on $\gamma$.

\begin{lemma}
\label{lemmaHk}
Let $A\in\Rrnn$ be a matrix for which relation~\eqref{fov} holds
and let $H_k$ be the matrix obtained in the SAI Krylov subspace 
method, see \eqref{Arn_sai},\eqref{Hk_sai}. Then there exists
a constant $\omega_k\geqs0$ such that 
\begin{equation}
\label{stb_Hk}  
\|\exp(-tH_k)\|\leqs e^{-t\omega_k}.
\end{equation}
\end{lemma}

\begin{proof}
Let $\omega=\min_{x\in\Cc^n,\|x\|=1}\Re(x^*Ax)$.
It is well known (see, e.g., \cite[Theorem~2.4]{HundsdorferVerwer:book}) 
that
$$
\Re(x^*Ax) \geqs\omega,\;\forall x\in\Cc^n \quad\Leftrightarrow\quad 
\|\exp(-tA)\|\leqs e^{-t\omega}.
$$
According to~\eqref{fov}, these two equivalent relations hold with
$\omega\geqs 0$.
Furthermore, these conditions are equivalent to
\cite[Theorem~2.13]{HundsdorferVerwer:book}
$$
\|(I+\gamma A)^{-1}\|\leqs \frac1{1+\gamma\omega},
$$   
which holds for all $\gamma>0$ and all $\omega\in\Rr$ 
provided that $1+\gamma\omega>0$.  
We assume $\gamma>0$, as is the case 
in the SAI Krylov subspace method.
Let us define 
$$
\omega_k := \frac1\gamma(\|\tilde{H}_k\|^{-1}-1),
$$
so that $\|\tilde{H}_k\|=1/(1+\gamma\omega_k)$.
We have
$$
\frac1{1+\gamma\omega_k}=\|\tilde{H}_k\| = \|V_k^T(I+\gamma A)^{-1}V_k\|
\leqs\|(I+\gamma A)^{-1}\|\leqs\frac1{1+\gamma\omega},
$$
from which it follows that $0\leqs\omega\leqs\omega_k$.
Since $\tilde{H}_k=(I+\gamma H_k)^{-1}$ (cf.\ \eqref{Hk_sai}),
we obtain, again using \cite[Theorem~2.13]{HundsdorferVerwer:book},
\begin{equation}
\label{stb_Hk1}  
\|(I+\gamma H_k)^{-1}\|=\frac1{1+\gamma\omega_k}\leqs\frac1{1+\gamma\omega}
\quad\Leftrightarrow\quad 
\Re(x^*H_kx)\geqs \omega_k,\;\forall x\in\Cc^k,
\end{equation}
which is equivalent to inequality~\eqref{stb_Hk} we prove.
\end{proof}

Define function $\varphi(z)$ as (see, e.g.,~\cite{HochbruckOstermann2010})
\begin{equation}
\label{phi}
\varphi(z)=(e^z-1)/z.  
\end{equation}

\begin{proposition}
\label{p:rk_sai}
Let $A\in\Rrnn$ be a matrix for which relation~\eqref{fov} holds
and let $r_k(t)$ be the
residual of the SAI Krylov subspace method defined by~\eqref{rk_sai},
applied to solve problem~\eqref{expA}.
Then it holds for all $t\geqs 0$
\begin{align}
\label{rk_sai1}
r_k(t) &= \beta_k(t) \cdot w_{k+1}, \quad 
\beta_k(t)=\frac{\tilde{h}_{k+1,k}}{\gamma}e_k^T (I+\gamma H_k) u(t),
\quad w_{k+1}=(I+\gamma A)v_{k+1},
\\
\notag
\|r_k(t)\|
&= |\beta_k(t)|\, \|w_{k+1}\|,
\\
\notag
|\beta_k(t)| &\leqs
\beta\tilde{h}_{k+1,k}\left(
\frac{1}{\gamma} \min\left\{
t\|(I+\gamma H_k)H_k\|\, \varphi(-t\omega_k)\, ,\,
\|I+\gamma H_k\| (1+e^{-t\omega_k} )
\right\} \right.
\\
\label{rk_est}
& \qquad\qquad\qquad\qquad\qquad\qquad\qquad\qquad
  \qquad\qquad\qquad\qquad\qquad
+ |h_{k,1}|
\biggr )
\end{align}
where $u(t)$ is defined by \eqref{expH},\eqref{Hk_sai},
$\omega_k\geqs0$ is the constant introduced in~\eqref{stb_Hk}
and $\tilde{h}_{k+1,k}$ and $h_{k,1}$ are the corresponding
entries of the matrices $\Hht_k\in\Rr^{(k+1)\times k}$ and
$H_k\in\Rr^{k\times k}$, respectively (see~\eqref{Arn_sai},\eqref{Hk_sai}).
Here, the minimum is taken among the two elements of the set
indicated by the braces $\{\dots\}$.
Note that $\Hht_k$ in the estimate above depends on $\gamma$
(and, hence, so do $H_k$, $u(t)$, and $\omega_k$).
\end{proposition}

\begin{proof}
Relation~\eqref{rk_sai1} is identical to~\eqref{rk_sai} (cf.~\eqref{Hk_sai}),
its proof can be found
in~\cite{BGH13}.  Using \eqref{Hk_sai},\eqref{ivp1}, we have
$$
\begin{aligned}
|\beta_k(0)| &=\frac{\tilde{h}_{k+1,k}}{\gamma}|e_k^T\Ht_k^{-1}u(0)|=
\frac{\tilde{h}_{k+1,k}}{\gamma}|e_k^T(I+\gamma H_k)\beta e_1|
\\
&=\tilde{h}_{k+1,k}|e_k^T H_k\beta e_1|=\beta\tilde{h}_{k+1,k}|h_{k,1}|.
\end{aligned}
$$
Furthermore, 
it is not difficult to check that, cf.~\eqref{phi},
$$
u(t)-u(0) = (\exp(-t H_k) - I )u(0) = -tH_k\varphi(-tH_k)u(0).
$$
Therefore, with $u(0)=\beta e_1$ and $\Ht_k^{-1}=I+\gamma H_k$,
we can estimate
\begin{equation}
\label{est_ut1}
\begin{aligned}
&\|(I+\gamma H_k)(u(t)-u(0)\| =
t\|(I+\gamma H_k)H_k\varphi(-tH_k)u(0)\|
\\
&\leqs
t\|(I+\gamma H_k)H_k\| \|\varphi(-tH_k)\| \|u(0)\|\leqs
\beta t\|(I+\gamma H_k)H_k\| \varphi(-t\omega_k).
\end{aligned}
\end{equation}
Here we used an inequality
$$
\|\varphi(-tH_k)\|\leqslant \varphi(-t\omega_k),
$$
which holds due to the property~\eqref{stb_Hk},
\eqref{stb_Hk1}, 
see, e.g.,~\cite[proof of Lemma~2.4]{HochbruckOstermann2010}.
The estimate~\eqref{est_ut1} is especially useful for small $t$.
An alternative estimate, which may be sharper for large $t$,
is
\begin{equation}
\label{est_ut2}
\begin{aligned}
&\| (I+\gamma H_k)(u(t)-u(0)\| =
  \|(I+\gamma H_k) (\exp(-tH_k)-I)u(0)\|
\\
&\leqs
\|I+\gamma H_k\| \|\exp(-tH_k)-I\| \|u(0)\|
\leqs
\beta\|I+\gamma H_k\|(1+\|\exp(-tH_k)\|)
\\
&\leqs
\beta\|I+\gamma H_k\|(1+e^{-t\omega_k}),
\end{aligned}
\end{equation}
where the estimate~\eqref{stb_Hk} is used.
From~\eqref{est_ut1},\eqref{est_ut2} it follows that
$$
\| (I+\gamma H_k)(u(t)-u(0)\|\leqs\beta\min
\left\{
t\|(I+\gamma H_k)H_k\| \varphi(-t\omega_k),
\|I+\gamma H_k\|(1+e^{-t\omega_k})
\right\}.
$$
We can then estimate
$$
\begin{aligned}
& |\beta_k(t) |\leqs|\beta_k(t)-\beta_k(0)| + |\beta_k(0)|
\\
&=
\frac{\tilde{h}_{k+1,k}}{\gamma}|e_k^T(I+\gamma H_k)(u(t)-u(0))|
+ \beta\tilde{h}_{k+1,k}|h_{k,1}|
\\
&\leqs
\frac{\tilde{h}_{k+1,k}}{\gamma}\| (I+\gamma H_k)(u(t)-u(0)\|
+ \beta\tilde{h}_{k+1,k}|h_{k,1}|
\\
&\leqs
\frac{\tilde{h}_{k+1,k}}{\gamma}
\beta\min
\left\{
t\|(I+\gamma H_k)H_k\| \varphi(-t\omega_k),
\|I+\gamma H_k\|(1+e^{-t\omega_k})
\right\}
+ \beta\tilde{h}_{k+1,k}|h_{k,1}|
\\ &=
\beta\tilde{h}_{k+1,k}\left(\frac{1}{\gamma}
\min\left\{
t\|(I+\gamma H_k)H_k\| \varphi(-t\omega_k),
\|I+\gamma H_k\|(1+e^{-t\omega_k})
\right\}
+ |h_{k,1}|\right),
\end{aligned}
$$
which proves~\eqref{rk_est}.
\end{proof}

We note that the estimate~\eqref{rk_est} is, unfortunately, far from
sharp to reflect the noticed dependence of the residual $r_k(t)$
on $\gamma$
(cf.\ Figures~\ref{f:rk} and~\ref{f:rk1}).
However, the following remark should be made.

\begin{remark}
Numerical experiments show that the value $|h_{k,1}|$
(recall $|\beta_k(0)|=\beta\tilde{h}_{k+1,k}|h_{k,1}|$)
appearing in~\eqref{rk_est} is usually small, typically many orders
of magnitude smaller than the other term $\frac1{\gamma}\min\{\dots\}$
appearing in
the right hand side of~\eqref{rk_est}.
If $\beta_k(0)=0$ then \eqref{rk_est} formally shows
that for any $k$ and any tolerance $\varepsilon>0$ a time
interval $[0,\delta]$ can be found such that
$\|r_k(s)\|\leqs\varepsilon$ for $s\in[0,\delta]$.
In this case the RT and AccuRT restarting strategies
are guaranteed to converge for any restart length.
Of course, such a $\delta$ can still
be too small to be used in practice,
so that adjusting $\gamma$, as is done in AccuRT,
may be necessary to make the restarting practical.
\end{remark}

\begin{figure}
  \centering{\texttt{\begin{minipage}{0.8\linewidth}
\% Given: $A\in\Rrnn$, $v\in\Rr^n$, $t>0$, $k_{\max}$ and $\tol>0$ \\
convergence $:=$ false \\
$\gammachanged := \texttt{false}$\\
carry out LU factorization $LU:=I+\gamma A$\\
while (not(convergence) and $t>0$)  \\\hspace*{2em}
   $\beta := \|v\|$                  \\\hspace*{2em}
   $v_1 := v/\beta$                   \\\hspace*{2em}
   for $k=1,\dots,k_{\max}$          \\\hspace*{4em}
       if $\gammachanged$                                  \\\hspace*{6em}
          solve $(I+\gamma A)w = v_k$ iterative,         \\\hspace*{8em}
               preconditioned by $LU$                      \\\hspace*{4em}
       else\\\hspace*{6em}
          solve $(I+\gamma A)w = v_k$ by LU factorization\\\hspace*{4em}
       end                                               \\\hspace*{4em}
       for $i=1,\dots,k$                                 \\\hspace*{6em}
           $\tilde{h}_{i,k} := w^Tv_i$                      \\\hspace*{6em} 
	   $w      := w - \tilde{h}_{i,k}v_i$               \\\hspace*{4em}  
       end                                               \\\hspace*{4em}  
       $h_{k+1,k} := \|w\|$                                 \\\hspace*{4em}     
       $H_k := \frac1\gamma ( \tilde{H}_k^{-1} - I)$        \\\hspace*{4em}
       compute $u(s_j)$, $\|r_k(s_j)\|$, $s_j=jt/3$, $j=1,2,3$
                                                         \\\hspace*{4em}
       resnorm $:= \max_j\|r_k(s_j)\|$                     \\\hspace*{4em}
       if resnorm $\leqs\tol$ and $k>1$                  \\\hspace*{6em} 
          convergence $:=$ true                           \\\hspace*{6em} 
          break loop for $k=\dots$                       \\\hspace*{4em} 
       elseif $k=k_{\max}$  ~~~~~ 
          \% --- restart at step $k_{\max}$           \\\hspace*{6em}
          compute $\|r_k(s_j)\|$, $s_j=jt/500$, $j=1,\dots,500$   \\\hspace*{6em}
          $r_{\min} := \min_j\|r_k(s_j)\|$                   \\\hspace*{6em}
          if $r_{\min} > \tol$                             \\\hspace*{8em} 
              $\delta := 0$                                \\\hspace*{8em} 
              $\gamma         := \gamma/2$                 \\\hspace*{8em} 
              $\gammachanged := \texttt{true}$             \\\hspace*{6em} 
          else                                            \\\hspace*{8em}
              $\delta := \max\{ s_j \; | \; \|r_k(s_j)\|\leqs\tol\}$
                                                          \\\hspace*{6em}  
          end                                             \\\hspace*{6em} 
          $u := \exp(-\delta H_k) e_1$                      \\\hspace*{6em}  
          $v := V_k (\beta u)$                               \\\hspace*{6em}
          $t    := t -\delta$                               \\\hspace*{6em} 
       end                                                 \\\hspace*{4em} 
       $v_{k+1} := w/h_{k+1,j}$                                \\\hspace*{2em} 
   end                                                     \\
end                                                        \\
$y_k := V_k (\beta u(s_3))$
\end{minipage}}}
  \caption{Description of the AccuRT restarting algorithm.
  The algorithm computes SAI Krylov subspace approximation $y_k(t)\approx\exp(-tA)v$ such that
  for its residual $r_k(t)$ holds $\|r_k(s)\|\leqslant\tol$ for $s=t/3$, $s=2t/3$ and $s=t$.}
\label{f:alg}
\end{figure}

\subsection{Solving the shifted linear systems}
We now show that an LU factorization computed for the shifted
matrix $I+\gamma A$ can be successfully used as a preconditioner
for the shifted matrix $I+\gammat A$ with an adjusted shift $\gammat$
such that $0<\gammat\leqs\gamma$.
More precisely, the shifted linear system
$$
\Aa x = b, \qquad \Aa =I+\gammat A,
$$
is preconditioned as
\begin{equation}
\label{prec}  
\Mm^{-1}\Aa x =\Mm^{-1}b, \qquad \Mm =I+\gamma A.
\end{equation}
It is then not difficult to show (see Proposition~\ref{sai_syst} below)
that even simple Richardson iteration for the preconditioned
system~\eqref{prec}, namely,
\begin{equation}
\label{Rich}
x_{m+1}=\tilde{G}x_m + \Mm^{-1}b,  \qquad \tilde{G}=I-\Mm^{-1}\Aa,
\end{equation}
converge unconditionally, i.e., for the spectral radius $\rho(\tilde{G})$
of $\tilde{G}$ holds $\rho(\tilde{G})<1$.  Hence, the eigenvalues
of the preconditioned matrix $\Mm^{-1}\Aa$ are located
on the complex plane inside the unit circle centered
at point $z=1+0i$, $i^2=-1$.
This means that any other modern Krylov subspace method such as GMRES,
BiCGSTAB, QMR and other~\cite{templates,Henk:book,SaadBook2003} should successfully converge for the
preconditioned linear system~\eqref{prec}.

On the other hand, the smaller the shift value $\gammat$,
the better the shifted matrix $I+\gammat A$ is conditioned.
Hence, for a small $\gammat$ it may turn out
that an unpreconditioned iterative method converges
fast enough.
Therefore, in Proposition~\ref{sai_syst} we give
a sufficient condition which guarantees that the preconditioned
Richardson method converges faster than unpreconditioned
one.

\begin{proposition}
\label{sai_syst}
Let $0<\gammat\leqs\gamma$ and a linear system with the matrix
$I+\gammat A$ is solved iteratively.  Then Richardson iteration~\eqref{Rich} with
the preconditioner matrix $I+\gamma A$ converges.

Furthermore, assume  unpreconditioned Richardson iteration
converges, too.
Then Richardson iteration with the preconditioner matrix $I+\gamma A$ 
converges faster than
unpreconditioned Richardson iteration provided that
\begin{equation}
\label{convRich}
\frac{1}{1+\gamma \rho(A)}<\frac\gammat\gamma,  
\end{equation}
where $\rho(A)$ is the spectral radius of the matrix $A$.
\end{proposition}

\begin{proof}
Let $\lambda$ be an eigenvalue of $A$.
The eigenvalues of the preconditioned matrix $(I+\gamma A)^{-1}(I+\gammat A)$ read
$$
\frac{1+\gammat\lambda}{1+\gamma\lambda}=
1-(1-\frac\gammat\gamma)\frac{\gamma\lambda}{1+\gamma\lambda}.
$$
The preconditioned Richardson iteration converge if and only if all
the eigenvalues of the iteration matrix $\tilde{G}=I-(I+\gamma A)^{-1}(I+\gammat A)$
are smaller in absolute value than 1, i.e.,
$$
\left|(1-\frac\gammat\gamma)\frac{\gamma\lambda}{1+\gamma\lambda}\right|<1.
$$
The left-hand side of this inequality can be bounded as 
$$
\left|(1-\frac\gammat\gamma)\frac{\gamma\lambda}{1+\gamma\lambda}\right|\leqs
\frac{|\gamma\lambda|}{|1+\gamma\lambda|}<1,
$$
where the last inequality holds because all the eigenvalues of $A$
have nonnegative real parts (see~\eqref{fov}).
Hence, the preconditioned Richardson method converges.

Furthermore note that the unpreconditioned Richardson iteration matrix
reads $G=I-(I+\gammat A)=-\gammat A$.
The preconditioned Richardson iteration converges faster than
unpreconditioned one provided that $\rho(\tilde{G})<\rho(G)$, i.e.,
$$
(1-\frac\gammat\gamma)\max_{\lambda}
\left|\frac{\gamma\lambda}{1+\gamma\lambda}\right|
<\gammat\max_{\lambda}|\lambda|=\gammat\rho(A). 
$$
The left-hand side here can be bounded by $1-\frac\gammat\gamma$
and we see that the inequality holds as soon as
$$
(1-\frac\gammat\gamma)<\gammat\rho(A).
$$
It is easy to check that the last inequality is equivalent to~\eqref{convRich}.
\end{proof}

\section{Numerical experiments}
\label{s:exper}
\subsection{Experiment setup and details}
We implemented the AccuRT algorithm as shown in Figure~\ref{f:alg}
with two small modifications.  First, as soon as $\gamma$ is decreased
in the restart phase of the algorithm (see the line $\gamma:=\gamma/2$), 
we restrict the time interval on which the residual 
norm minimum is searched for on the next restart from $[0,t]$ to $[0,t/2]$.
This is done to account for the fact that the residual norm 
$\|r_k(s)\|$ is likely not to be small for $s>t/2$, see Figure~\ref{f:rk1}.
Once a restart is successful, i.e., $\delta>0$ and the time interval is
decreased (line $t:=t-\delta$), we set the search interval back to $[0,t]$. 
Another small modification is that, as an option, GMRES(10) can be used as a linear
solver instead of the LU factorization, not only when $\gamma$ is decreased.
In this case the ILU($\varepsilon$) preconditioner can be used which is 
computed only once and re-used for all values of $\gamma$.
To decide whether to use the preconditioner or not, Proposition~\ref{sai_syst}
can be used.

Initial value for $\gamma$ can be provided to our AccuRT subroutine as 
an optional parameter.  By default, if $\gamma$ is not provided by the user,
it is set to $t/20$.  Note that $\gamma=t/10$ is 
suggested in~\cite{EshofHochbruck06} an appropriate value for 
moderate tolerances $\approx 10^{-6}$ for symmetric matrices.
Setting $\gamma$ to $t/20$ appears to
be a reasonable choice because, as our limited experience suggests, optimal
values of $\gamma$ for nonsymmetric matrices are usually smaller than $t/10$.

In the experiments below, within the framework of the SAI Krylov subspace
method, 
we compare the performance of the AccuRT restarting 
with that of the RT restarting.
The RT restarting has been recently compared with three other
restarting strategies, namely, the EXPOKIT restarting~\cite{EXPOKIT}, 
the Niehoff--Hochbruck restarting~\cite{PhD_Niehoff} and
the residual restarting~\cite{CelledoniMoret97,BGH13}.  
The presented numerical tests are performed in Matlab on a Linux PC
with 8~CPUs Intel Xeon~E5504 2.00GHz.  

\subsection{Convection--diffusion problem}
\label{s:t1}
The matrix $A$ in this problem is obtained by a standard five point
central-difference discretization
of a convection--diffusion operator defined on 
functions $u(x,y)$, with $(x,y)\in \Omega=[0,1]\times [0,1]$, 
and $u|_{\partial\Omega}=0$.  The operator reads
\begin{gather*}
  L[u]=-(D_1u_x)_x-(D_2u_y)_y + \Pe{}\,\left(
  \frac12(v_1u_x + v_2u_y) + \frac12((v_1u)_x + (v_2u)_y) \right)
\\
D_1(x,y)=\begin{cases}
    10^3 \; &(x,y)\in [0.25,0.75]^2, \\
    1       &\text{otherwise},
    \end{cases}\qquad\quad D_2(x,y)=\frac12 D_1(x,y),
\\
v_1(x,y) = x+y,\qquad v_2(x,y)=x-y.
\end{gather*}
Here the special way the convective terms are written in
takes care that the discretized convection terms
give a skew-symmetric matrix~\cite{Krukier79}.
In the experiments, we use a uniform
$802\times 802$ grid and the Peclet numbers $\Pe=200$
and $\Pe=1000$.
The problem size for this grid is
$n=800^2=640\,000$.
For both Peclet numbers we have $\|\frac12(A+A^T)\|_2\approx 6000$,
whereas $\|\frac12(A-A^T)\|_2\approx 0.5$ for $\Pe=200$
and $\|\frac12(A-A^T)\|_2\approx 2.5$ for $\Pe=1000$.
Hence, in both cases the matrices are weakly non-symmetric.
The values of the function $\sin(\pi x)\sin(\pi y)$ on the 
finite-difference grid are assigned to the initial vector $v$,
which is then normalized as $v:=v/\|v\|$.
The final time is set to $t=1$.

In this test the initial value of $\gamma$ in the AccuRT restarting 
is not altered from its default value $t/20$, whereas 
the RT restarting uses the usual value $t/10$.  This does not
necessarily gives an advantage to AccuRT because optimal
$\gamma$ values detected by AccuRT are smaller than $t/20$ anyway.

Results for this test problem are presented in Table~\ref{t:c-d}.
As can be seen in the first two lines of the table, 
the RT restarting does not yield a better accuracy as the tolerance
gets smaller.  The AccuRT implementation with the same sparse LU
factorization as the linear solver (see line~3 of the table)
is able to give the required accuracy, although the CPU time is
increased by a factor of 10.
Note that this CPU time measurement (done in MAT LAB) 
is not very representative: it does not correspond to the number 
of steps (77 versus 30).
This is because direct solvers (LU~factorization
and the backslash operator) are implemented in MAT LAB quite 
efficiently, whereas the iterative solvers not.
Hence, in the test problem the CPU time for 
applying the sparse LU factorization as a preconditioner
within the GMRES(10) solver turns out to be relatively high.
However, once the algorithm
detects a suitable value of $\gamma$, this value can be
used for other initial vectors $v$: in this case, as seen in
line~4 of the table, we get the same high accuracy for a moderate
increase in the CPU time.

Furthermore, we evaluate our approach on this test problem
with an iterative inner solver.  For this purpose we used GMRES(10)
with the ILUT($\varepsilon=10^{-3}$) preconditioner.
Line~5 of the table says
the RT restarting requires now 37 SAI-Arnoldi steps.
This is because at the
end of the third restart the residual turns out to be just above
the tolerance $10^{-8}$ (and it is just below the tolerance when
the direct inner solver is used, thus resulting in 30~steps,
already seen in line~3 of the table).
AccuRT with the same inner iterative setting requires twice as long 
CPU time but the accuracy is now improved, see line~6.
Finally, once the AccuRT has detected a proper value of $\gamma$,
its costs get well below the costs of the RT restarted algorithm.

In the lower part of Table~\ref{t:c-d} the results for 
a higher Peclet number are presented.  For the tolerance $10^{-6}$
the RT restarting gives a result with an accuracy {\tt1.17e-06}.
This seems to be fine.  However, for the tolerance $10^{-7}$
the attained accuracy is the same, while the computational costs
increase from 16 to 23~steps.  In the following two lines 
of the table we see the results for the AccuRT restarting.
It yields an improved accuracy for the increased CPU time
(78.6~s instead of 63.2~s).  As seen in the next table
line, once a proper value of $\gamma$
is detected, the same accuracy can be obtained within less
CPU time, namely 47.9~s.

\begin{table}
\caption{Results for the convection--diffusion test problem, 
$802\times 802$ grid, final time $t=1$.}
\label{t:c-d}
\centering\begin{tabular}{lccc}
\hline\hline
Method & Tolerance,    &  CPU      & Steps (inner \\ 
       & error         & time (s)  & iterations)  \\
\hline
\multicolumn{4}{c}{$\Pe=200$, restart length 10}\\
RT, sparse LU   & {\tt1e-06}, {\tt2.50e-07} & 46.2 & 20 (---)    \\
RT, sparse LU   & {\tt1e-08}, {\tt2.59e-07} & 48.7 & 30 (---)    \\
AccuRT, sparse LU, GMRES(10)
                & {\tt1e-08}, {\tt1.35e-08} & 515  & 77 (1022) \\
AccuRT, sparse LU, GMRES(10), detected $\gamma$
                & {\tt1e-08}, {\tt1.38e-08} & 58.2 & 57 (---)  \\
RT, GMRES(10)/ILUT
                & {\tt1e-08}, {\tt2.58e-07} & 94.1  & 37 (454) \\
AccuRT, GMRES(10)/ILUT
                & {\tt1e-08}, {\tt1.85e-08} & 193  & 77 (1258) \\
AccuRT, GMRES(10)/ILUT, detected $\gamma$
                & {\tt1e-08}, {\tt1.51e-08} & 69.8 & 57 (342)  \\
\hline
\multicolumn{4}{c}{$\Pe=1000$, restart length 8}\\
RT, GMRES(10)/ILUT  & {\tt1e-06}, {\tt1.17e-06} & 52.3 & 16 (176)\\
RT, GMRES(10)/ILUT  & {\tt1e-07}, {\tt1.17e-06} & 66.1 & 23 (253)\\
AccuRT, GMRES(10)/ILUT & {\tt1e-06}, {\tt3.58e-07} & 79.3 & 35 (356)\\
AccuRT, GMRES(10)/ILUT, detected $\gamma$ 
                    & {\tt1e-06}, {\tt3.07e-07} & 47.9 & 27 (174)\\
\hline
\multicolumn{4}{c}{$\Pe=1000$, restart length 7}\\
RT, GMRES(10)/ILUT  & {\tt1e-06}, {\tt4.11e-06} & 60.6 & 21 (231)\\
AccuRT, GMRES(10)/ILUT & {\tt1e-06}, {\tt1.47e-06} & 44.2 & 17 (136)\\
AccuRT, GMRES(10)/ILUT, perturbed $\gamma$ 
                    & {\tt1e-06}, {\tt1.47e-06} & 65.0 & 24 (274)\\
\hline
\end{tabular}
\end{table}

\subsection{Maxwell's equations in a lossless medium}
Consider the Maxwell equations for a three-dimensional domain 
in a lossless and source-free medium:
\begin{equation}
\label{mxw}
\begin{aligned}
  \frac{\partial\bm{H}}{\partial t} &= -\frac{1}{\mu}\nabla\times\bm{E},
  \\
\frac{\partial\bm{E}}{\partial t} &= \frac{1}{\varepsilon}\nabla\times\bm{H},
\end{aligned}
\end{equation}
Here
$\varepsilon$ and $\mu$ are scalar functions of $(x,y,z)$
called permittivity and permeability, respectively,
whereas magnetic field $\bm{H}$ and 
electric field $\bm{E}$ are unknown vector-valued functions of $(x,y,z,t)$.
The boundary conditions assign zero to the tangential electric field
components.
Physically this can mean either
perfectly conducting domain boundary or the so-called 
``far field condition''~\cite{Taflove,Kole01}.
%
The problem setup is adopted from the second test in~\cite{Kole01}:
in a spatial domain 
$[-6.05, 6.05]\times [-6.05, 6.05]\times [-6.05, 6.05]$
filled with air (relative permittivity $\varepsilon_r=1$) 
a dielectric specimen with relative permittivity $\varepsilon_r=5.0$
is placed which occupies the region
$[-4.55, 4.55]\times [-4.55, 4.55]\times [-4.55, 4.55]$.
In the specimen there are 27~spherical voids ($\varepsilon_r=1$)
of radius~$1.4$, whose centers are positioned at 
$(x_i,y_j,z_k)=(3.03 i, 3.03j, 3.03 k)$, $i,j,k= -1,0,1$.
The initial values for all the components of both fields
$\bm{H}$ and $\bm{E}$ are set to zero, except for the
$x$- and $y$-components of $\bm{E}$.
These two have nonzero values in the middle of the
spatial domain to represent a light emission.
The standard finite-difference staggered Yee discretization in space
leads to an ODE system of the form~\eqref{ivp}. 
The spatial meshes used in the test comprise either $40\times 40\times 40$
or $80\times 80\times 80$ grid Yee cells and
lead to problem size $n=413\,526$ or 
$n=3\,188\,646$, respectively.
After the discretization is carried out, the initial value vector
$v\in\Rr^n$ is normalized as $v:=v/\|v\|$.
Comparison of the results obtained for the two meshes
shows that this spatial resolution should be sufficient for this test.
The final time is set to $t=1$.

This test represents is a tough problem for the SAI Krylov method
because the matrix $A$ is strongly nonsymmetric (in fact, a diagonal
matrix $D$ can be chosen such that $D^{-1}AD$ is skew-symmetric).
For strongly nonsymmetric problems, such as the discretized lossless 
Maxwell equations,
SAI Krylov methods are likely to be inefficient \cite{VerwerBotchev09}
and, indeed, other restarted Krylov
subspace exponential schemes are reported to be more efficient
for this test problem~\cite{BoKn2020}.  Furthermore, it is 
a three-dimensional vector problem where 6~unknowns ($x$, $y$ and $z$
vector components for each field) are associated with each 
computational cell.  Hence, depending on the specific parameter 
values, solving linear systems with the shifted matrix $I+\gamma A$
may not be a trivial task.
However, for this particular test setting 
it turns out that $\gamma\|A\|$ is small
enough so that condition~\eqref{convRich} does not hold and even unpreconditioned Richardson
iteration can successfully be used for solving the shifted linear systems.
In the test runs, we use unpreconditioned GMRES(10) iterative solver. 
Therefore, 
we include this test to show capabilities of the proposed AccuRT restarting.

\begin{figure}
\begin{center}
\includegraphics[width=0.8\textwidth]{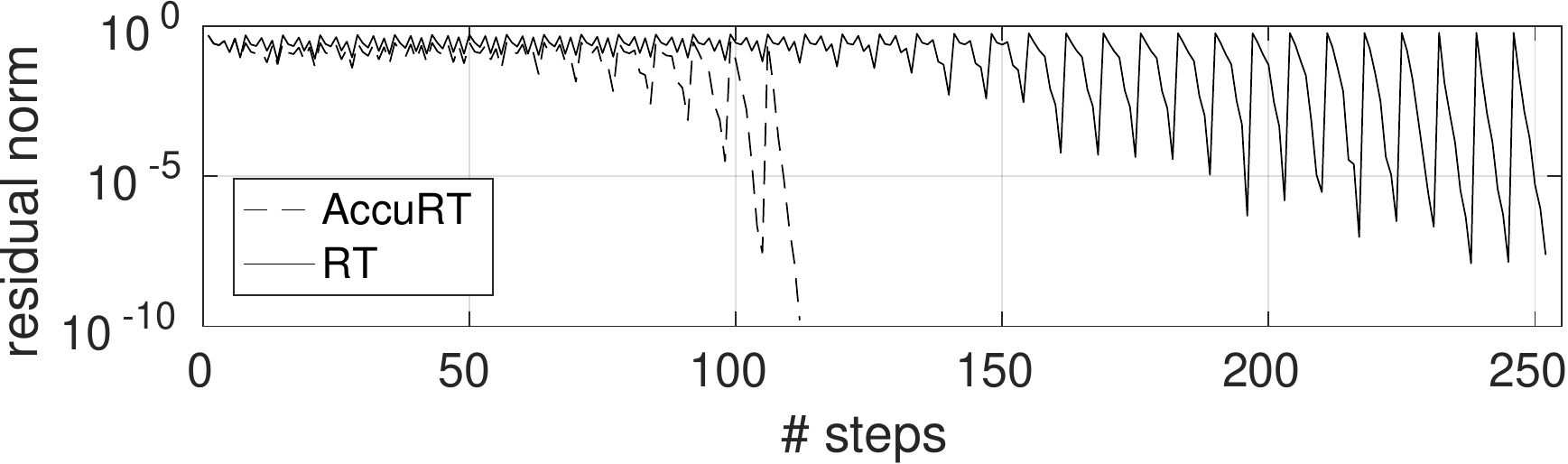}
\\
\includegraphics[width=0.8\textwidth]{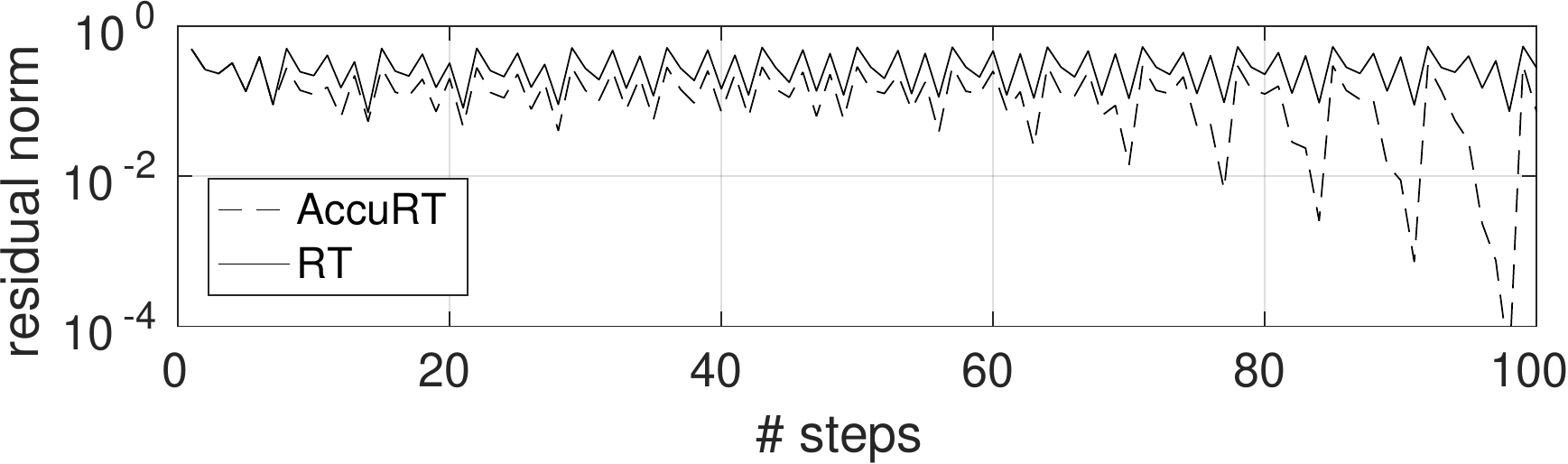}
\end{center}
\caption{Convergence of the SAI Krylov method with RT (solid line) and AccuRT
  (dashed line) restarting for lossless Maxwell equations test,
  mesh $40\times 40\times 40$, restart length 7. The bottom
  plot is a close up of the top plot.  Each zigzag corresponds to
a restart.}
\label{f:mxw}
\end{figure}

Our experience indicates that,
to have a fast convergence with SAI Krylov methods
for strongly nonsymmetric matrices $A$,
$\gamma$ should be taken significantly 
smaller than the commonly used value $t/10$~\cite{Botchev2016}.  
Hence we set $\gamma$ to $t/80=1/80$
for both RT and AccuRT methods, and the AccuRT restarting
can eventually further decrease this value.  
The results are presented 
in Table~\ref{t:mxw} and in Figures~\ref{f:mxw} and~\ref{f:eff}.
The AccuRT restarting clearly outperforms the RT restarting
both in terms the costs and the attained accuracy.
The accuracy loss occurs at the first restarts when
the residual norm turns out to be nowhere within the required
tolerance.  The AccuRT then reduces $\gamma$ which not only repairs
the accuracy loss but also leads to a faster solution of the shifted
systems (the smaller $\gamma$, the faster unpreconditioned GMRES
converges).  Moreover, the reduced $\gamma$ leads then to a further
efficiency gain by AccuRT.
As can be seen in Figure~\ref{f:eff}, it is obtained due to
larger intervals $[0,\delta]$ on which the residual appears to
be smaller than the required tolerance (recall that the time interval
at each restart reduces from $[0,t]$ to $[0,t-\delta]$).

\begin{table}
\caption{Results for the lossless Maxwell equation test problem, final time $t=1$.}
\label{t:mxw}
\centering\begin{tabular}{lccc}
\hline\hline
Method & Tolerance,    &  CPU      & Steps (inner \\ 
       & error         & time (s)  & iterations)  \\
\hline
\multicolumn{4}{c}{mesh $40\times 40\times 40$, restart length 7}\\
RT, GMRES(10)   & {\tt1e-09}, {\tt2.97e-07} & 138.1 & 252 (2268)\\
AccuRT, GMRES(10)  & {\tt1e-09}, {\tt6.60e-08} & 50.3 & 112 (798) \\
\hline
\multicolumn{4}{c}{mesh $80\times 80\times 80$, restart length 8}\\
RT, GMRES(10)   & {\tt1e-09}, {\tt7.20e-07} & 1888  & 312 (3781)\\
AccuRT, GMRES(10)  & {\tt1e-09}, {\tt3.92e-08} &  785 & 191 (1410)\\
\hline
\end{tabular}
\end{table}

\begin{figure}
\begin{center}
\includegraphics[width=0.8\textwidth]{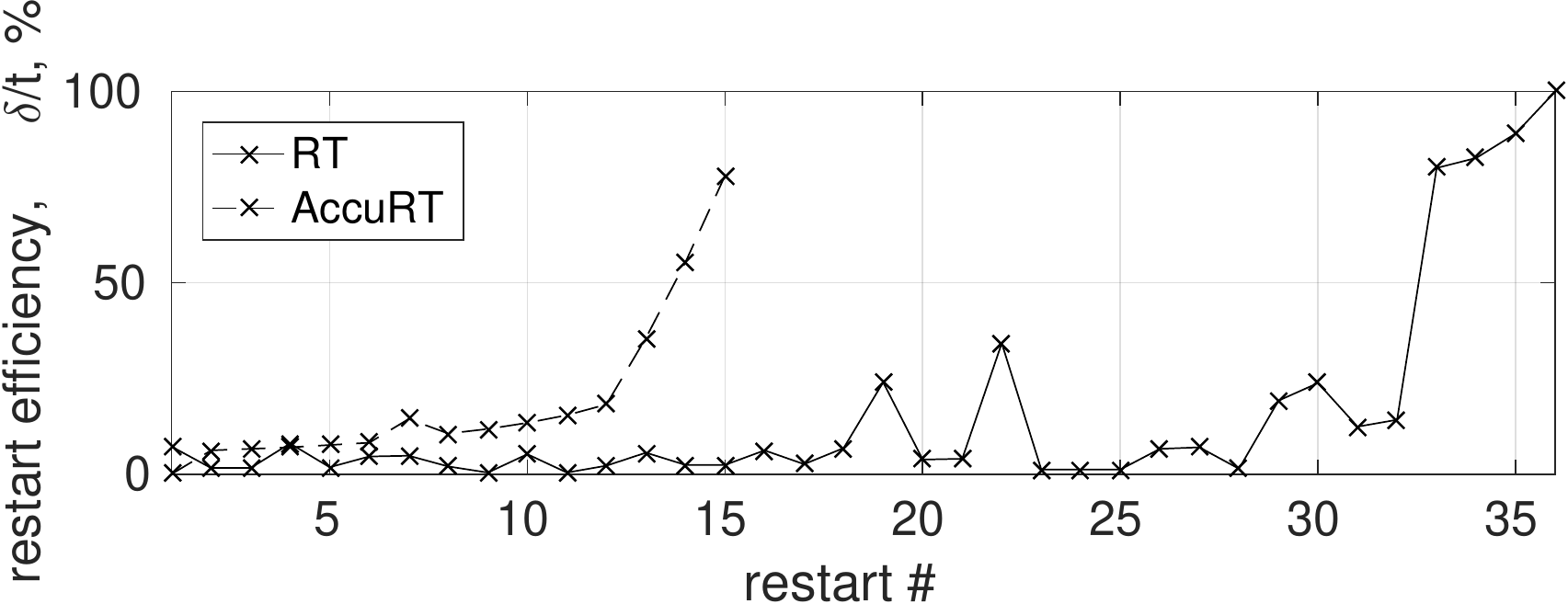}
\end{center}
\caption{Restarting efficiency as a ratio of the decreased
  time interval length $\delta$ and total remaining time interval length $t$
for the lossless Maxwell equations test,
mesh $40\times 40\times 40$, restart length 7. Efficiency $0\%$ at
the first restart means that $\gamma$ is decreased by AccuRT.}
\label{f:eff}
\end{figure}

\section{Conclusions}
\label{s:concl}
The presented AccuRT (accurate residual--time) restarting seems 
to be a promising approach
for SAI (shift-and-invert) Krylov subspace methods evaluating the matrix
exponential of nonsymmetric matrices.  It has all the attractive
properties of the standard RT (residual--time) restarting and allows to avoid its 
accuracy loss while preserving efficiency of the method.

Several research directions for further studies can be indicated.
First, the minimum search of the residual norm is currently carried out
on a uniform set of points spread over the time interval.  This could 
also be done on a nonuniform grid refined in the regions
where the residual norm has its local minima.  An adaptive procedure
for building such a mesh could be designed.
Furthermore, a question on how to extend this approach
to symmetric matrices could be investigated.  We hope
to address these issues in the future.

\bibliography{matfun,my_bib,mxw}
\bibliographystyle{abbrv}
\end{document}